\documentclass[11pt,a4paper]{article}

\usepackage[utf8]{inputenc}
\usepackage[T1]{fontenc}
\usepackage{lmodern}
\usepackage[margin=1in]{geometry}
\usepackage{microtype}

\usepackage{amsmath,amssymb,amsfonts,amsthm,mathtools}
\usepackage{bm}
\usepackage{mathrsfs}
\usepackage{enumitem}
\usepackage{graphicx}
\usepackage{xcolor}
\usepackage{hyperref}
\usepackage[nameinlink,capitalize,noabbrev]{cleveref}

\hypersetup{
    colorlinks=true,
    linkcolor=blue!60!black,
    citecolor=blue!60!black,
    urlcolor=blue!60!black
}

\usepackage{aliascnt}

\theoremstyle{plain}
\newtheorem{theorem}{Theorem}[section]

\newaliascnt{proposition}{theorem}
\newtheorem{proposition}[proposition]{Proposition}
\aliascntresetthe{proposition}

\newaliascnt{lemma}{theorem}
\newtheorem{lemma}[lemma]{Lemma}
\aliascntresetthe{lemma}

\newaliascnt{corollary}{theorem}

\aliascntresetthe{corollary}

\newaliascnt{assumption}{theorem}
\newtheorem{assumption}[assumption]{Assumption}
\aliascntresetthe{assumption}

\theoremstyle{definition}
\newaliascnt{definition}{theorem}
\newtheorem{definition}[definition]{Definition}
\aliascntresetthe{definition}

\newaliascnt{remark}{theorem}
\newtheorem{remark}[remark]{Remark}
\aliascntresetthe{remark}

\newaliascnt{example}{theorem}

\aliascntresetthe{example}

\newcommand{\R}{\mathbb{R}}
\newcommand{\N}{\mathbb{N}}
\newcommand{\cH}{\mathcal{H}}
\newcommand{\cM}{\mathcal{M}}
\newcommand{\cP}{\mathcal{P}}
\newcommand{\cG}{\mathcal{G}}
\newcommand{\cA}{\mathcal{A}}
\newcommand{\cF}{\mathcal{F}}
\newcommand{\cE}{\mathcal{E}}
\newcommand{\cJ}{\mathcal{J}}

\newcommand{\supp}{\operatorname{supp}}
\newcommand{\diver}{\operatorname{div}}

\newcommand{\KL}{\operatorname{KL}}

\newcommand{\stim}{\mathrm{stim}}
\newcommand{\react}{\mathrm{react}}
\newcommand{\dyn}{\mathrm{dyn}}

\usepackage{authblk}

\title{\bf Branched Optimal Transport for Stimulus to Reaction Brain Mapping}

\author[1]{Cristian Mendico}
\affil[1]{Institut de Math\'ematique de Bourgogne - UMR 5584 CNRS, Universit\'e Bourgogne Europe \\ \href{mailto:cristian.mendico@u-bourgogne.fr}{cristian.mendico@u-bourgogne.fr}
}
\date{\today}

\begin{document}

\maketitle

\begin{abstract}
A central problem in systems neuroscience is to determine how an external stimulation is propagated through the brain so as to produce a reaction. Current deterministic and stochastic control models quantify transition costs between brain states on a prescribed network, but do not treat the transport network itself as an unknown. Here we propose a variational framework in which the inferred object is a graph/current connecting a stimulation source measure to a reaction target measure. The model is posed as an anisotropic branched optimal transport problem, where concavity of the flux cost promotes aggregation and branching. The support of an optimal current defines a stimulus-to-reaction routing architecture, interpreted as a brain reaction map. We prove existence of minimizers in discrete and continuous formulations and introduce a hybrid stochastic extension combining ramified transport with a path-space Kullback--Leibler control cost on the induced graph dynamics. This approach provides a mathematical mechanism for inferring propagation architectures rather than controlling trajectories on fixed substrates.

\medskip 

\noindent {\bf Key words:} Branched Optimal Transport ; Mathematical Biology ; Variational models.

\medskip

\noindent {\bf MSC 2020:} 35Q49 - 35Q92 - 58E25.
\end{abstract}

\tableofcontents

\section{Introduction}

A fundamental problem in systems neuroscience is to understand how an externally applied stimulus is propagated through the brain so as to produce a reaction. Depending on the experimental setting, the reaction may be motor, perceptual, autonomic, or cognitive; in all cases, however, one is faced with the same conceptual question: \emph{through which neural routing architecture is the stimulus conveyed to the reaction-producing configuration?} While modern neuroimaging and connectomics provide increasingly detailed information on structural and functional connectivity, they do not by themselves identify a mathematically canonical \emph{map of propagation} from stimulation to response. The aim of this paper is to propose such a map as the solution of a variational problem.

A large body of recent work addresses related questions through control-theoretic models of brain dynamics. In these approaches, one typically fixes an ambient dynamical system, deterministic or stochastic, on a prescribed state space or network, and studies the cost of driving the system from an initial state to a target state by means of admissible control inputs \cite{Gu2015,Kawakita2022,Kamiya2023}. This perspective has been highly successful in quantifying controllability properties of structural brain networks, in measuring the difficulty of task-dependent state transitions, and in incorporating stochasticity through entropic optimal control and Schr\"odinger bridge formulations. Mathematically, the optimization variable in those models is a control process, or equivalently a controlled trajectory or controlled path-space law. The underlying network---or state-space geometry---is fixed in advance.

More broadly, the present work fits within a long tradition of mathematical biology and mathematical neuroscience devoted to linking neural structure, collective dynamics, and large-scale communication through explicit models. At mesoscopic scales, neural mass and neural field models have long provided a mathematical language for describing spatially extended cortical activity, including traveling waves, localized bumps, oscillations, and pattern formation \cite{Breakspear2017,Bressloff2012,Pinotsis2014}. Their stochastic counterparts have also been placed on a rigorous analytical footing, with results on well-posedness, gradient-flow structure, invariant measures, and delayed or diffusive extensions that are directly relevant for coarse-grained descriptions of brain activity \cite{FaugerasInglis2015,KuehnTolle2019,SpekKuznetsovVanGils2020}, see also \cite{GalvesLoecherbachPouzat2024}. At larger scales, whole-brain models combine structural connectivity with local population dynamics and noise in order to study how anatomy constrains functional coordination, effective connectivity, and task-dependent reconfiguration \cite{Gilson2020,DecoKringelbach2020,PathakRoyBanerjee2022}. Recent work has further emphasized that realistic macroscale modeling of the brain requires integrating network topology with local biological annotations and with quantitative measures of structure--function coupling \cite{BazinetHansenMisic2023,FotiadisEtAl2024}. 

Our approach is complementary to this literature. Rather than starting from a fixed connectome and studying the resulting dynamics, we take the source--target pair and the anatomical transport cost as primary inputs, and seek the routing architecture itself as the solution of a variational problem.
 In many situations, the primary quantity of interest is not the energetic or informational cost of realizing a transition on a known substrate, but rather the \emph{substrate itself}. If one knows where the stimulus is injected, and one has a mathematical representation of the reaction-producing configuration, can one infer a preferred transport architecture connecting the two? Differently, can one formulate a model in which the output of the optimization problem is not a controlled path, but a \emph{graph}---or, in the continuous setting, a \emph{rectifiable current}---describing the routing of the stimulus through the brain? This is the perspective adopted here.

The key observation is that such a question is naturally geometric rather than purely dynamical. If information or activation propagates jointly through nearby neural pathways before branching toward different downstream regions, then the corresponding cost should not depend only on endpoint displacements or on state trajectories, but on the actual \emph{network supporting the flux}. This leads directly to the mathematical paradigm of \emph{branched optimal transport}. In branched transport, one minimizes a transportation cost that is concave with respect to the transported flux. Concavity favors aggregation: transporting mass together for part of the route is strictly cheaper than transporting it separately all the way. As a consequence, minimizers are not diffuse flows but branching structures. This phenomenon was first formalized in the discrete Gilbert model, where one minimizes over weighted directed graphs satisfying Kirchhoff balance laws \cite{Gilbert1967}. It was later extended by Xia to a continuous Eulerian formulation in terms of divergence-constrained rectifiable vector measures, or equivalently one-dimensional currents \cite{Xia2003}. These models, together with their Lagrangian traffic-plan counterparts, now form a well-established mathematical theory of ramified transport \cite{BernotCasellesMorel2009,Santambrogio2007}.

Our proposal is to adapt this framework to stimulus-to-reaction brain mapping. We represent the external stimulation by a source measure
\[
\mu_{\stim}^+ \in \mathcal M_+(\Omega),
\]
and the reaction-producing neural configuration by a target measure
\[
\mu_{\react}^- \in \mathcal M_+(\Omega),
\]
with equal total mass, on a compact domain $\Omega\subset\mathbb R^d$ representing the relevant anatomical or mesoscopic neural space. The unknown is a transport current
\[
v=\tau\theta\,\mathcal H^1\llcorner M,
\]
where $M$ is a countably $1$-rectifiable set, $\tau$ is an orientation field, and $\theta$ is a scalar multiplicity encoding transported flux. The current is constrained by the balance law
\[
\operatorname{div} v = \mu_{\stim}^+ - \mu_{\react}^-,
\]
and is selected by minimizing an anisotropic branched transport energy of the form
\[
\mathcal M_{\alpha,c}(v)
=
\int_M c(x,\tau(x))\,\theta(x)^\alpha\,d\mathcal H^1(x),
\qquad 0<\alpha<1.
\]
Here the exponent $\alpha$ encodes the ramification incentive, while the coefficient
\[
c:\Omega\times\mathbb S^{d-1}\to(0,\infty)
\]
weights transportation according to local anatomical or functional features, such as tissue geometry, preferred orientation, conductivity, or tractographic information. The support of an optimal current is then interpreted as the \emph{brain reaction map}: the preferred routing architecture that conveys the stimulation toward the reaction.

This point of view differs conceptually from standard control formulations in one essential respect. In deterministic or stochastic network control, the geometry is prescribed and the optimization determines how best to move \emph{on} that geometry. In the present framework, the geometry itself is inferred \emph{from} the source--target pair by variational selection. The current or graph is the primary optimization variable; controls and trajectories, when they appear, are secondary objects built on top of that geometry. This distinction is not merely formal. It reflects a different modeling objective: not the cost of using a known neural substrate, but the inference of an effective routing architecture from stimulation and response data.

At the same time, the branched transport picture does not aim to replace dynamical control models. Rather, it provides a geometric backbone onto which dynamical models may subsequently be imposed. This motivates the hybrid extension developed later in the paper. Once an admissible graph/current has been selected, one may associate with it a stochastic neural dynamics and then quantify the discrepancy between uncontrolled and controlled evolutions on that graph through a path-space Kullback--Leibler cost. More precisely, if $(G, w)$ denotes an admissible weighted graph and $\mathbb P^{G,w,0}$ and $\mathbb P^{G,w,u}$ denote the uncontrolled and controlled path-space laws induced by a stochastic dynamics on $G$, one may define a secondary dynamic term
\[
\cJ_{\dyn}(G,w)
:=
\inf_u
\KL\!\left(\mathbb P^{G,w,u}\,\middle\|\,\mathbb P^{G,w,0}\right).
\]
In this way, a graph may be penalized not only by its ramified geometric transport cost, but also by the minimum informational effort required to realize a compatible controlled stochastic evolution on it. The resulting model preserves the graph/current as the main unknown while incorporating a dynamic regularization principle inspired by Schr\"odinger bridge theory \cite{Kawakita2022,Kamiya2023}.


From a mathematical viewpoint, the paper develops this idea at two levels. First, we formulate discrete and continuous branched transport models adapted to the neuroscience setting and prove existence of minimizers under natural assumptions. Second, we propose a stochastic extension in which the geometric optimizer is coupled to a path-space control cost. From a modeling viewpoint, the framework offers a candidate variational definition of a brain reaction map, combining the geometric efficiency of ramified transport with the flexibility of stochastic control. We believe that this dual perspective may provide a useful bridge between geometric measure theory, optimal transport, and systems neuroscience. Numerical simulations associated with the model we are proposing in this work will appear in {\it Multimodal branched transport infers anatomically aligned brain reaction maps} by the author.

\paragraph{Structure of the paper.}

 \cref{sec:prelim} fixes the notation and the admissible data.
 In \cref{sec:estimation} we explain how these quantities may be modeled and estimated in realistic settings. \cref{sec:discrete} introduces the discrete graph formulation and proves its basic well-posedness properties. \cref{sec:continuous} develops the continuous Xia-type current formulation. \Cref{sec:hybrid} presents a stochastic extension inspired by Schr\"odinger bridge control. Finally, in \cref{sec:interpretation} we discuss the neuroscientific interpretation of the minimizer and its connection with data-driven estimation.

\section{Notations}\label{sec:prelim}


Let $\Omega \subset \R^d$ be a compact connected domain representing the brain, with $d=3$ in a volumetric model or $d=2$ in a cortical-surface model.

We denote by $\cP(\Omega)$ the space of Borel probability measures on $\Omega$ and by $\cM(\Omega;\R^d)$ the space of finite $\R^d$-valued Radon measures on $\Omega$.

\begin{definition}[Stimulus and reaction measures]
A \emph{stimulus source measure} is a nonnegative finite Borel measure
\[
\mu_{\stim}^+ \in \cM(\Omega), \qquad \mu_{\stim}^+(\Omega)=m>0,
\]
representing the entry distribution of the external stimulation. A \emph{reaction target measure} is a nonnegative finite Borel measure
\[
\mu_{\react}^- \in \cM(\Omega), \qquad \mu_{\react}^-(\Omega)=m,
\]
representing the neural distribution associated with the production of the reaction. The pair $(\mu_{\stim}^+,\mu_{\react}^-)$ is called \emph{balanced} if
\[
\mu_{\stim}^+(\Omega)=\mu_{\react}^-(\Omega).
\]
\end{definition}

The balanced condition is the natural conservation law for the transport problem. The common mass $m$ may be normalized to $1$ without loss of generality.

In the brain, the transport cost should not depend only on Euclidean length. One should also penalize propagation through regions or directions that are anatomically implausible or metabolically expensive.

\begin{definition}[Anatomical cost density]
An \emph{anatomical cost density} is a continuous function
\[
c:\Omega\times \mathbb{S}^{d-1}\to (0,\infty)
\]
such that there exist constants $0<c_*\leq c^*<\infty$ with
\[
c_* \leq c(x,\tau) \leq c^* \qquad \forall\; (x,\tau)\in \Omega\times \mathbb{S}^{d-1}.
\]
\end{definition}

The costs $c(x,\tau)$ may encode structural connectivity, tractography constraints, tissue anisotropy, or metabolic penalties.




Throughout the paper we fix
\[
\alpha \in (0,1).
\]
The concavity of $w\mapsto w^\alpha$ is what drives the branching effect: grouping two flows before they split is cheaper than transporting them independently over the same distance.

\section{Estimation of $\mu_{\stim}^{+}$, $\mu_{\react}^{-}$, and $c(x,\tau)$ from data}\label{sec:estimation}

This section explains how the abstract ingredients of the model may be estimated from empirical neuroimaging and tractography data. Our goal is not to prescribe a unique preprocessing pipeline, but rather to indicate a mathematically consistent interface between the variational formulation and realistic data analysis.

\subsection{General principle}

The source measure $\mu_{\stim}^{+}$ is meant to represent the neural entry distribution of the external stimulation, while the target measure $\mu_{\react}^{-}$ represents the spatial distribution of activity associated with the production of the reaction. In practice, both measures can be estimated either at the level of regions of interest (ROI-level atomic measures) or at the voxel/source-space level (absolutely continuous or kernel-smoothed measures).

A common template is the following. Let $(R_i)_{i=1}^N$ be a brain partition into ROIs with representative points $x_i\in\Omega$, and let $a_i^{\stim},a_i^{\react}\ge 0$ be nonnegative activity scores extracted from early and late temporal windows, respectively. Then one sets
\[
\mu_{\stim}^{+}
=
m\sum_{i=1}^N
\frac{a_i^{\stim}}{\sum_{j=1}^N a_j^{\stim}}\,
\delta_{x_i},
\qquad
\mu_{\react}^{-}
=
m\sum_{i=1}^N
\frac{a_i^{\react}}{\sum_{j=1}^N a_j^{\react}}\,
\delta_{x_i},
\]
provided the denominators are nonzero. If a continuous representation is preferred, one replaces each Dirac mass $\delta_{x_i}$ by a smooth kernel $\varphi_i\ge 0$ supported in $R_i$ and normalized by $\int_\Omega \varphi_i(x)\,dx=1$.

The anatomical cost density $c(x,\tau)$ should encode the local ease or difficulty of information propagation. In applications, it may be estimated from diffusion MRI, tractography, structural connectivity, or other anatomical priors.

This reduction of neuroimaging data to source and target measures is consistent with the standard interface used in large-scale brain modeling, where parcellated regional activity and structural priors provide the natural bridge between empirical measurements and mesoscale mathematical models \cite{Gilson2020,Patow2024}. In the same spirit, recent work on biologically annotated connectomes and macroscale structure--function coupling emphasizes that regional summaries should be interpreted jointly with anatomical and dynamical constraints, rather than as purely descriptive observables \cite{BazinetHansenMisic2023,FotiadisEtAl2024}.

\subsection{Estimation from fMRI and EEG/MEG data}

For fMRI data, let $y_i(t)$ denote the preprocessed BOLD signal in ROI $R_i$, and let
\[
\widetilde y_i(t)=y_i(t)-\bar y_i^{\,\mathrm{base}}
\]
be the baseline-corrected signal.

In task-based fMRI, task-evoked responses are routinely summarized through BOLD contrasts and generalized-linear-model-based statistics, which provide a natural source of nonnegative regional scores for the construction of $\mu_{\stim}^{+}$ and $\mu_{\react}^{-}$ \cite{Friston2009,EstebanEtAl2020}. More generally, task-based fMRI is explicitly designed to map the neural response to cognitive, perceptual, or motor manipulations through BOLD fluctuations, making early and late activation windows a natural proxy for stimulus-entry and reaction-related activity, respectively \cite{EstebanEtAl2020}.

Early stimulus-related scores may then be defined by
\[
a_i^{\stim}
=
\int_{t_0}^{t_1}\bigl(\widetilde y_i(t)\bigr)_+\,dt,
\]
or, more generally, by a positive task-evoked contrast such as a GLM coefficient or a $t$-statistic. Likewise, late reaction-related scores may be defined by
\[
a_i^{\react}
=
\int_{t_2}^{t_3}\bigl(\widetilde y_i(t)\bigr)_+\,dt.
\]
These quantities are then inserted into the general normalization formula above.

For EEG or MEG data, after source localization one obtains regional source time series $u_i(t)$. In this case, the same construction can be applied with
\[
a_i^{\stim}
=
\int_{t_0}^{t_1}|u_i(t)|\,dt,
\qquad
a_i^{\react}
=
\int_{t_2}^{t_3}|u_i(t)|\,dt,
\]
or with signed variants if one wishes to distinguish excitatory increases from other effects.

For EEG and MEG, source reconstruction provides spatially resolved time series associated with cortical or subcortical generators, thereby offering a natural route to define regional activity scores at much finer temporal resolution than fMRI \cite{KnoscheHaueisen2022,AsadzadehEtAl2020}. This high temporal resolution is particularly advantageous when one aims to distinguish early stimulus-driven components from later reaction-related components, which is precisely the temporal separation required by the present source--target construction \cite{KnoscheHaueisen2022}.

\begin{remark}
If deactivations are considered relevant, one may either treat positive and negative parts separately or introduce multi-source/multi-target signed decompositions. In the present paper we restrict attention to positive measures, which is the natural setting for branched transport.
\end{remark}

\subsection{Estimation of the anatomical cost density}

We now indicate several possible constructions of the anatomical cost density $c(x,\tau)$.

The use of diffusion MRI and tractography to construct anatomical transport costs is consistent with a broad connectomic literature in which white-matter architecture is used to estimate structural connectivity and to constrain large-scale communication models \cite{HagmannEtAl2007,ZhangEtAl2022,YehEtAl2021}. In particular, both tract-specific and connectome-based analyses routinely extract quantitative information from tractography that can be incorporated into spatially heterogeneous or direction-dependent weights, exactly of the kind encoded here by $c(x,\tau)$ \cite{ZhangEtAl2022,YehEtAl2021}. More generally, recent work on biologically annotated connectomes and structure--function coupling suggests that macroscale brain models should integrate wiring information with local biological or geometric annotations, which further supports the use of anisotropic and biologically informed transport costs in the present framework \cite{BazinetHansenMisic2023,FotiadisEtAl2024}.

A simple isotropic heterogeneous model is obtained from a smooth white-matter score $w:\Omega\to[0,1]$ by setting
\[
c(x,\tau)
=
c_{\min}+(c_{\max}-c_{\min})(1-w(x)),
\qquad
0<c_{\min}<c_{\max},
\]
so that propagation through white matter is cheaper than propagation through less favorable regions.

A direction-dependent model can be derived from a continuous field of symmetric nonnegative matrices $D(x)\in\R^{d\times d}$, for instance estimated from diffusion MRI or tractography, by defining
\[
c(x,\tau)
=
\sqrt{\tau^\top\bigl(D(x)+\varepsilon I\bigr)^{-1}\tau},
\qquad
\varepsilon>0.
\]
This makes transport cheaper along directions aligned with strong local diffusion or fiber orientation.

A more flexible mixed model is
\[
c(x,\tau)=a(x)\sqrt{\tau^\top A(x)\tau}+b(x),
\]
where $a,b:\Omega\to(0,\infty)$ are continuous and bounded above and below by positive constants, and $A(x)$ is a uniformly positive definite symmetric matrix field.

Finally, if one has a continuous directional plausibility score
\[
p:\Omega\times\mathbb S^{d-1}\to[0,1],
\]
one may set
\[
c(x,\tau)=c_*+(c^*-c_*)\bigl(1-p(x,\tau)\bigr),
\]
so that more plausible propagation directions carry lower transport cost.

\begin{remark}
The regularization term $\varepsilon I$ in tensor-based models guarantees uniform positive definiteness and therefore ensures the continuity and positive upper/lower bounds required in the definition of anatomical cost density.
\end{remark}

\section{Discrete graph model}\label{sec:discrete}

\subsection{Weighted directed graphs}

We first formulate a discrete model, since it directly returns a graph and is particularly convenient for computational implementations.

\begin{definition}[Embedded weighted directed graph]
An \emph{embedded weighted directed graph} in $\Omega$ is a quadruple
\[
G=(V,E,w)
\]
with the following components:
\begin{enumerate}[label=($\roman*$)]
    \item $V=\{v_1,\dots,v_N\}\subset \Omega$ is a finite set of vertices;
    \item $E$ is a finite set of oriented edges;
    \item for each $e\in E$, the edge is represented by an injective Lipschitz path
    \[
    \gamma_e:[0,\ell_e]\to \Omega,
    \]
    with $\gamma_e(0)=v_{e^-}$ and $\gamma_e(\ell_e)=v_{e^+}$, where $e^- ,e^+\in \{1,\dots,N\}$ denote the tail and head of the edge;
    \item $w=(w_e)_{e\in E}$ with $w_e\geq 0$ is the family of transported fluxes.
\end{enumerate}
\end{definition}

Vertices may represent observed ROIs, latent relay regions, or free Steiner-type branching points.

\begin{definition}
For each edge $e\in E$, define the anatomical-geometric cost
\[
\beta_e := \int_0^{\ell_e} c\!\left(\gamma_e(s),\frac{\dot\gamma_e(s)}{|\dot\gamma_e(s)|}\right)\, ds,
\]
whenever $\dot\gamma_e(s)\neq 0$ almost everywhere.
\end{definition}

Let $b:V\to \R$ be a prescribed vertex supply/demand function satisfying
\[
\sum_{v\in V} b(v)=0.
\]
Heuristically, we interpret:
\begin{enumerate}
    \item $b(v)>0$: stimulus source at $v$;
    \item $b(v)<0$: reaction sink at $v$;
    \item $b(v)=0$: pure relay/branching vertex.
\end{enumerate}

\begin{definition}
A weighted directed graph $(G,w)$ is \emph{Kirchhoff-admissible} with respect to $b$ if for every vertex $v\in V$,
\[
\sum_{e\in \mathrm{Out}(v)} w_e - \sum_{e\in \mathrm{In}(v)} w_e = b(v).
\]
\end{definition}

This is the discrete mass-conservation law.

\begin{definition}
Given $\alpha\in(0,1)$, the \emph{discrete branched brain transport energy} of an admissible weighted graph $(G,w)$ is
\[
\cE_\alpha(G,w)
:=
\sum_{e\in E} \beta_e\, w_e^\alpha.
\]
\end{definition}

The problem of interest is therefore:
\begin{equation}\label{eq:discrete-problem}
\min \left\{ \cE_\alpha(G,w) : (G,w)\in \cG_{\mathrm{adm}} \right\},
\end{equation}
where $\cG_{\mathrm{adm}}$ is a prescribed admissible family of embedded weighted directed graphs satisfying the Kirchhoff constraints and the source/sink specification induced by $(\mu_{\stim}^+,\mu_{\react}^-)$.

\begin{remark}
The support of a minimizer $(G^*,w^*)$ is interpreted as the \emph{brain reaction map}. The edge multiplicities $w_e^*$ quantify the amount of routed information (or effective signal load), while the branching structure indicates shared pathways and integrative hubs.
\end{remark}

\subsection{Finite-library existence theorem}

The broadest version of the discrete problem involves optimization over both graph topology and geometry. We first consider the important case in which one starts from a finite candidate library of possible edges.

\begin{assumption}[Finite admissible library]
There is a finite directed graph
\[
G_0=(V_0,E_0)
\]
embedded in $\Omega$, and $\cG_{\mathrm{adm}}$ consists of all weighted directed subgraphs of $G_0$ satisfying the Kirchhoff balance law for a prescribed $b:V_0\to \R$.
\end{assumption}

\begin{theorem}\label{thm:discrete-existence}
Under the finite-library assumption, the optimization problem \eqref{eq:discrete-problem} admits at least one minimizer.
\end{theorem}

\begin{proof}
Since the candidate edge set $E_0$ is finite, an admissible weighted subgraph is completely determined by a nonnegative vector
\[
w=(w_e)_{e\in E_0}\in [0,\infty)^{|E_0|},
\]
with the convention that edges with $w_e=0$ are inactive.

The Kirchhoff constraints are linear:
\[
A w = b,
\]
for a suitable incidence matrix $A$. Hence the admissible set
\[
\cA:=\{w\in [0,\infty)^{|E_0|}: Aw=b\}
\]
is closed and convex.

Because the total outgoing mass from source vertices is fixed, every admissible $w$ has uniformly bounded $\ell^1$ norm. Indeed, summing the positive supplies gives
\[
\sum_{e\in E_0} w_e \leq C(b,G_0),
\]
for some constant depending only on the graph incidence structure and the source/sink vector $b$. Hence $\cA$ is bounded. Since it is also closed in a finite-dimensional space, it is compact.

The objective
\[
w \mapsto \sum_{e\in E_0} \beta_e w_e^\alpha
\]
is continuous on $[0,\infty)^{|E_0|}$. Therefore, by the Weierstrass theorem, it attains its minimum on $\cA$.
\end{proof}

\begin{proposition}\label{prop:cycle}
Assume that $\beta_e>0$ for every edge $e$. Then every minimizer of \eqref{eq:discrete-problem} can be chosen cycle-free.
\end{proposition}

\begin{proof}
Suppose that an admissible minimizer contains a directed cycle $C$ carrying strictly positive flow. Let
\[
\varepsilon := \min_{e\in C} w_e >0.
\]
Reduce the flow by $\varepsilon$ on every edge of the cycle, leaving all other edge weights unchanged. The Kirchhoff constraints remain satisfied, since every vertex on the cycle loses the same incoming and outgoing amount.

Because $t\mapsto t^\alpha$ is strictly increasing on $(0,\infty)$ and each $\beta_e>0$, the total energy strictly decreases:
\[
\sum_{e\in C}\beta_e\bigl((w_e-\varepsilon)^\alpha - w_e^\alpha\bigr)<0.
\]
This contradicts minimality. Hence a minimizer may always be chosen without cycles.
\end{proof}

\begin{remark}
\Cref{prop:cycle} shows that the optimal graph has a tree-like flavor, as expected in branched transportation. In the present application, such a cycle-free minimizer is naturally interpreted as a preferred routing tree for the stimulus-to-reaction propagation.
\end{remark}

\section{Continuous current formulation}\label{sec:continuous}

The discrete model is often the most intuitive. However, when one wishes to optimize not only weights but also the geometric placement of branches and paths, it is more natural to pass to a continuous Eulerian formulation.

\subsection{Rectifiable transport currents}

\begin{definition}
A vector measure $v\in \cM(\Omega;\R^d)$ is called \emph{rectifiable} if there exist:
\begin{enumerate}[label=($\roman*$)]
    \item a countably $1$-rectifiable set $M\subset \Omega$,
    \item a measurable unit tangent field $\tau:M\to \mathbb{S}^{d-1}$,
    \item a multiplicity function $\theta:M\to [0,\infty)$ integrable with respect to $\cH^1\llcorner M$,
\end{enumerate}
such that
\[
v = \tau\,\theta\, \cH^1\llcorner M.
\]
Here $M$ is the unknown transport network, $\tau$ its orientation, and $\theta$ the flow multiplicity.
\end{definition}

\begin{definition}
Given balanced measures $(\mu_{\stim}^+,\mu_{\react}^-)$, an admissible current is a rectifiable vector measure $v$ satisfying
\[
\diver v = \mu_{\stim}^+ - \mu_{\react}^-
\]
in the sense of distributions on $\Omega$.
\end{definition}

\begin{definition}
Let $\Omega\subset\R^d$ be compact. Given balanced finite nonnegative measures
\[
\mu_{\stim}^+,\mu_{\react}^-\in \mathcal M(\Omega),
\qquad
\mu_{\stim}^+(\Omega)=\mu_{\react}^-(\Omega)=:m,
\]
we denote by $\mathcal A(\mu_{\stim}^+,\mu_{\react}^-)$ the class of acyclic rectifiable vector measures
\[
v=\tau\theta\,\mathcal H^1\llcorner M\in\mathcal M(\Omega;\R^d)
\]
such that
\[
\operatorname{div} v=\mu_{\stim}^+-\mu_{\react}^-.
\]
\end{definition}

\subsection{Continuous branched transport functional}

For a rectifiable current
\[
v = \tau\theta\,\cH^1\llcorner M,
\]
define
\[
\cM_{\alpha,c}(v)
:=
\int_M c(x,\tau(x))\,\theta(x)^\alpha\, d\cH^1(x),
\]
and set $\cM_{\alpha,c}(v)=+\infty$ if $v$ is not rectifiable.

The continuous brain mapping problem is:
\begin{equation}\label{eq:continuous-problem}
\min \left\{
\cM_{\alpha,c}(v):
v\in \cM(\Omega;\R^d),\ \diver v=\mu_{\stim}^+-\mu_{\react}^-
\right\}.
\end{equation}

\begin{remark}
The unknown support $M^*=\supp(v^*)$ of a minimizer $v^*$ is the continuous analogue of the optimal transport graph. It is the inferred brain routing architecture. The multiplicity $\theta^*$ quantifies the amount of information routed through each point of the network.
\end{remark}

\begin{lemma}\label{lem:continuous-lsc}
Let $\alpha\in(1-\frac1d,1)$. Assume that $\Omega\subset\R^d$ is compact and that
\[
c:\Omega\times\mathbb S^{d-1}\to(0,\infty)
\]
is continuous and satisfies
\[
0<c_*\le c(x,\tau)\le c^*<\infty
\qquad
\forall (x,\tau)\in\Omega\times\mathbb S^{d-1}.
\]
Let $(v_n)_n\subset \cA(\mu_{\stim}^{+},\mu_{\react}^{-})$ be a sequence such that
\[
\sup_n \cM_{\alpha,c}(v_n)<\infty
\]
and
\[
v_n \overset{*}{\rightharpoonup} v
\qquad\text{weakly-* in }\cM(\Omega;\R^d).
\]
Then
\[
v\in \cA(\mu_{\stim}^{+},\mu_{\react}^{-})
\]
and
\[
\cM_{\alpha,c}(v)\le \liminf_{n\to\infty}\cM_{\alpha,c}(v_n).
\]
\end{lemma}

\begin{proof}
The argument relies on the standard Eulerian--Lagrangian correspondence for branched transport, compactness of traffic plans, and lower semicontinuity of the corresponding Lagrangian $H$-mass.

Since
\[
v_n \overset{*}{\rightharpoonup} v
\qquad\text{in }\cM(\Omega;\R^d)
\]
and
\[
\diver v_n=\mu_{\stim}^{+}-\mu_{\react}^{-}
\qquad\forall\; n \in \N,
\]
the divergence constraint passes to the limit. Indeed, for every $\varphi\in C^1(\Omega)$,
\[
\int_\Omega \nabla\varphi\cdot dv
=
\lim_{n\to\infty}\int_\Omega \nabla\varphi\cdot dv_n
=
-\lim_{n\to\infty}\int_\Omega \varphi\,d(\mu_{\stim}^{+}-\mu_{\react}^{-})
=
-\int_\Omega \varphi\,d(\mu_{\stim}^{+}-\mu_{\react}^{-}),
\]
hence
\[
\diver v=\mu_{\stim}^{+}-\mu_{\react}^{-}.
\]

Moreover, since $c\ge c_*$, the sequence has uniformly bounded standard branched transport mass:
\[
c_*\int_{M_n}\theta_n^\alpha\,d\cH^1
\le
\int_{M_n} c(x,\tau_n(x))\,\theta_n^\alpha\,d\cH^1
=
\cM_{\alpha,c}(v_n),
\]
and therefore
\[
\sup_n \int_{M_n}\theta_n^\alpha\,d\cH^1<\infty.
\]

At this point we pass from the Eulerian formulation to the Lagrangian one. Since each $v_n$ belongs to $\cA(\mu_{\stim}^{+},\mu_{\react}^{-})$, it is in particular acyclic. By the standard correspondence between acyclic transport currents and good irrigation plans, there exists for every $n$ an irrigation plan $\eta_n$ whose associated current is exactly $v_n$ and such that the corresponding energy formula holds:
\[
\cM_{\alpha,c}(v_n)=\mathbf I_{\alpha,c}(\eta_n).
\]
Here
\[
\mathbf I_{\alpha,c}(\eta)
:=
\int_{\Gamma}\int_0^{+\infty}
c\!\left(\gamma(t),\frac{\dot\gamma(t)}{|\dot\gamma(t)|}\right)
|\dot\gamma(t)|\,\Theta_\eta(\gamma(t))^{\alpha-1}\,dt\,d\eta(\gamma),
\]
where $\Gamma$ denotes the space of 1-Lipschitz eventually constant curves in $\Omega$, and $\Theta_\eta(x)$ is the traffic multiplicity at $x$.
The equivalence between the Eulerian and Lagrangian formulations of branched transport, together with the corresponding energy identity, is standard; see \cite{BernotCasellesMorel2009,Pegon2017}.

Let
\[
m:=\mu_{\stim}^{+}(\Omega)=\mu_{\react}^{-}(\Omega).
\]
Since the multiplicity of a traffic plan is bounded above by the total transported mass, one has
\[
0\le \Theta_{\eta_n}(x)\le m
\qquad\forall x\in\Omega.
\]
Because $\alpha-1<0$, it follows that
\[
\Theta_{\eta_n}(x)^{\alpha-1}\ge m^{\alpha-1}.
\]
Using also the lower bound $c\ge c_*$, we obtain
\[
\mathbf I_{\alpha,c}(\eta_n)
\ge
c_*\,m^{\alpha-1}
\int_\Gamma\int_0^{+\infty}|\dot\gamma(t)|\,dt\,d\eta_n(\gamma)
=
c_*\,m^{\alpha-1}\int_\Gamma \ell(\gamma)\,d\eta_n(\gamma),
\]
where
\[
\ell(\gamma):=\int_0^{+\infty}|\dot\gamma(t)|\,dt.
\]
Hence
\[
\sup_n\int_\Gamma \ell(\gamma)\,d\eta_n(\gamma)<\infty.
\]

By the compactness theorem for traffic plans with uniformly bounded average length (see \cite[Chapter 3]{BernotCasellesMorel2009}), after extraction of a subsequence not relabelled, there exists a traffic plan $\eta$ such that
\[
\eta_n \rightharpoonup \eta
\qquad\text{narrowly in }\cP(\Gamma).
\]
Since the endpoint evaluation maps are continuous on the relevant compact subsets of $\Gamma$, the initial and terminal marginals are preserved in the limit, so that $\eta$ is an admissible irrigation plan from $\mu_{\stim}^{+}$ to $\mu_{\react}^{-}$. These facts are standard in the theory of traffic plans; see \cite{BernotCasellesMorel2009}.

Let $v_\eta$ denote the current induced by $\eta$. By continuity of the superposition map from traffic plans to vector-valued currents under narrow convergence and uniform length control, one has
\[
v_n \overset{*}{\rightharpoonup} v_\eta
\qquad\text{in }\cM(\Omega;\R^d).
\]
By uniqueness of the weak-* limit, it follows that
\[
v=v_\eta.
\]

We now use lower semicontinuity on the Lagrangian side. The functional $\mathbf I_{\alpha,c}$ is a weighted Lagrangian $H$-mass with $H(r)=r^\alpha$. The lower semicontinuity of the Lagrangian $H$-mass is proved in \cite{KrukowskiMarchese2026}; in the present setting the same argument applies to the continuous bounded anisotropic weight
\[
(x,\tau)\longmapsto c(x,\tau),
\]
and therefore yields
\[
\mathbf I_{\alpha,c}(\eta)
\le
\liminf_{n\to\infty}\mathbf I_{\alpha,c}(\eta_n).
\]
Combining this with the energy identity gives
\[
\cM_{\alpha,c}(v)
=
\cM_{\alpha,c}(v_\eta)
=
\mathbf I_{\alpha,c}(\eta)
\le
\liminf_{n\to\infty}\mathbf I_{\alpha,c}(\eta_n)
=
\liminf_{n\to\infty}\cM_{\alpha,c}(v_n).
\]

Finally, since $\mathbf I_{\alpha,c}(\eta)<\infty$, the associated current $v_\eta$ is rectifiable; see again \cite{BernotCasellesMorel2009,Pegon2017}. Therefore
\[
v=v_\eta\in \cA(\mu_{\stim}^{+},\mu_{\react}^{-}),
\]
and the proof is complete.
\end{proof}

We now state a standard well-posedness result in the spirit of branched transport theory.

\begin{theorem}\label{thm:continuous-existence}
Let $\alpha\in(1-\frac1d,1)$. Assume that:
\begin{enumerate}[label=(\roman*)]
    \item $\Omega\subset \R^d$ is compact and connected;
    \item $\mu_{\stim}^{+}$ and $\mu_{\react}^{-}$ are balanced finite nonnegative measures on $\Omega$;
    \item $c:\Omega\times\mathbb S^{d-1}\to(0,\infty)$ is continuous and there exist constants
    \[
    0<c_*\le c^*<\infty
    \]
    such that
    \[
    c_*\le c(x,\tau)\le c^*
    \qquad\forall (x,\tau)\in\Omega\times\mathbb S^{d-1}.
    \]
\end{enumerate}
Then the minimization problem
\[
\inf\Bigl\{
\cM_{\alpha,c}(v):v\in \cA(\mu_{\stim}^{+},\mu_{\react}^{-})
\Bigr\}
\]
admits at least one minimizer.
\end{theorem}

\begin{proof}
We first show that the admissible class is nonempty and that the infimum is finite. Since
\[
\alpha>1-\frac1d
\]
and the measures $\mu_{\stim}^{+}$ and $\mu_{\react}^{-}$ are balanced on the compact set $\Omega$, the classical irrigability theorem yields the existence of an admissible irrigation plan with finite $\alpha$-energy. Passing to the associated Eulerian current, one obtains some
\[
v_0\in \cA(\mu_{\stim}^{+},\mu_{\react}^{-})
\]
such that
\[
\cM_{\alpha,c}(v_0)\le c^*\,\cM_\alpha(v_0)<\infty.
\]
Hence the infimum is finite.

Let $(v_n)_n\subset \cA(\mu_{\stim}^{+},\mu_{\react}^{-})$ be a minimizing sequence:
\[
\cM_{\alpha,c}(v_n)\longrightarrow
\inf\Bigl\{
\cM_{\alpha,c}(v):v\in \cA(\mu_{\stim}^{+},\mu_{\react}^{-})
\Bigr\}.
\]
In particular, there exists $C>0$ such that
\[
\cM_{\alpha,c}(v_n)\le C
\qquad\forall\; n \in \N.
\]

Since the energy density is strictly positive, any cyclic component can be removed without changing the divergence and while decreasing the energy. Therefore, without loss of generality, we may assume that every $v_n$ is acyclic. By the correspondence with good irrigation plans and the compactness theorem for traffic plans with uniformly bounded average length, after extraction of a subsequence not relabelled there exists
\[
v\in \cM(\Omega;\R^d)
\]
such that
\[
v_n \overset{*}{\rightharpoonup} v
\qquad\text{in }\cM(\Omega;\R^d).
\]

We may now apply \cref{lem:continuous-lsc}. It follows that
\[
v\in \cA(\mu_{\stim}^{+},\mu_{\react}^{-})
\]
and
\[
\cM_{\alpha,c}(v)
\le
\liminf_{n\to\infty}\cM_{\alpha,c}(v_n).
\]
Since $(v_n)_n$ is minimizing, the right-hand side is exactly the infimum of the problem. Hence
\[
\cM_{\alpha,c}(v)
=
\inf\Bigl\{
\cM_{\alpha,c}(w):w\in \cA(\mu_{\stim}^{+},\mu_{\react}^{-})
\Bigr\},
\]
which proves that $v$ is a minimizer.
\end{proof}

\section{Hybrid stochastic extension}\label{sec:hybrid}

The branched transport model introduced above is purely geometric: it selects a routing architecture by minimizing a ramified transport cost between a stimulation source measure and a reaction target measure. In many situations, however, one also wishes to encode whether a given routing architecture is dynamically compatible with observed brain-state transitions. This motivates a hybrid extension in which each admissible weighted graph induces a stochastic neural dynamics, and the graph is penalized not only by its transport energy, but also by the minimum control effort required to realize prescribed initial and terminal distributions under that dynamics.

\subsection{Graph-induced linear stochastic dynamics}

We work in the finite-library discrete setting introduced in \cref{sec:discrete}. Let
\[
(G,w)\in \cG_{\mathrm{adm}}
\]
be an admissible weighted graph with vertex set
\[
V=\{v_1,\dots,v_n\}.
\]
We identify each vertex with one neural state variable, so that the neural activity is described by a vector
\[
X_t=(X_t^1,\dots,X_t^n)\in\R^n.
\]

\begin{definition}
Given an admissible weighted graph $(G,w)$, define the directed weighted adjacency matrix
\[
W_{G,w}=(W_{ij})_{1\le i,j\le n}
\]
by
\[
W_{ij}
:=
\sum_{\substack{e\in E\\ e^- = i,\ e^+ = j}} w_e,
\]
that is, $W_{ij}$ is the total flux carried by edges directed from $v_i$ to $v_j$.

Since the stochastic dynamics below is meant to represent effective propagation rather than purely directed mass balance, we introduce the symmetrized weighted adjacency
\[
S_{G,w}
:=
\frac12\bigl(W_{G,w}+W_{G,w}^\top\bigr).
\]
We then define the weighted degree matrix
\[
D_{G,w}
:=
\operatorname{diag}\!\left(\sum_{j=1}^n (S_{G,w})_{1j},\dots,\sum_{j=1}^n (S_{G,w})_{nj}\right),
\]
and the associated weighted graph Laplacian
\[
L_{G,w}:=D_{G,w}-S_{G,w}.
\]
\end{definition}

The matrix $L_{G,w}$ is symmetric and positive semidefinite. It is therefore natural to use it as the basic graph-induced diffusion operator in a linear neural dynamics.

\begin{definition}
Fix constants
\[
\kappa>0,\qquad \beta>0,\qquad \sigma_0>0,\qquad \sigma_1\ge 0.
\]
For each admissible weighted graph $(G,w)$, define
\begin{equation}\label{eq:canonical-A}
A_{G,w}
:=
-\kappa I_n-\beta L_{G,w},
\end{equation}
and
\begin{equation}\label{eq:canonical-C}
C_{G,w}
:=
\sigma_0 I_n+\sigma_1 D_{G,w}^{1/2}.
\end{equation}
\end{definition}

The interpretation is straightforward: the Laplacian term describes graph-mediated spreading of activity, the scalar term $-\kappa I_n$ models intrinsic relaxation toward baseline, and the matrix $C_{G,w}$ allows the amplitude of stochastic fluctuations to depend on the local weighted degree while preserving uniform nondegeneracy through $\sigma_0 I_n$.

We then consider the controlled linear stochastic differential equation
\begin{equation}\label{eq:graph-dynamics}
dX_t
=
A_{G,w}X_t\,dt
+
B_{\stim}a_t\,dt
+
u_t\,dt
+
C_{G,w}\,dW_t,
\qquad t\in[0,T],
\end{equation}
where:
\begin{enumerate}[label=(\roman*)]
    \item $X_t\in\R^n$ is the neural state vector;
    \item $a_t$ is a prescribed exogenous stimulus profile;
    \item $B_{\stim}\in\R^{n\times m_{\stim}}$ injects the stimulation into designated entry regions;
    \item $u_t$ is an adapted control process with values in $\R^n$;
    \item $W_t$ is a standard $n$-dimensional Brownian motion.
\end{enumerate}
The uncontrolled process corresponds to $u_t\equiv 0$.

\begin{proposition}\label{prop:A-stable}
For every admissible weighted graph $(G,w)$, the matrix $A_{G,w}$ defined in \eqref{eq:canonical-A} is symmetric negative definite. In particular, all its eigenvalues are strictly negative, and the uncontrolled linear system is exponentially stable.
\end{proposition}

\begin{proof}
Since $L_{G,w}$ is symmetric positive semidefinite, all its eigenvalues are nonnegative. Hence the eigenvalues of
\[
A_{G,w}=-\kappa I_n-\beta L_{G,w}
\]
are of the form
\[
-\kappa-\beta\lambda,
\qquad \lambda\ge 0.
\]
Since $\kappa>0$ and $\beta>0$, every such eigenvalue is strictly negative.
\end{proof}

\begin{proposition}\label{prop:C-elliptic}
For every admissible weighted graph $(G,w)$, the matrix $C_{G,w}$ defined in \eqref{eq:canonical-C} is symmetric positive definite. Moreover,
\[
C_{G,w}C_{G,w}^\top \ge \sigma_0^2 I_n
\qquad\text{for all }(G,w)\in\cG_{\mathrm{adm}}.
\]
\end{proposition}

\begin{proof}
Since $D_{G,w}$ is diagonal with nonnegative entries, $D_{G,w}^{1/2}$ is well defined and positive semidefinite. Therefore
\[
C_{G,w}=\sigma_0 I_n+\sigma_1 D_{G,w}^{1/2}
\]
is symmetric and all its eigenvalues are bounded below by $\sigma_0>0$. Hence $C_{G,w}$ is positive definite. The matrix inequality follows immediately.
\end{proof}

\begin{remark}
The maps
\[
(G,w)\longmapsto W_{G,w},\qquad
(G,w)\longmapsto S_{G,w},\qquad
(G,w)\longmapsto D_{G,w},\qquad
(G,w)\longmapsto L_{G,w}
\]
are continuous in the finite-dimensional edge-weight parameterization of the admissible graph library. Consequently,
\[
(G,w)\longmapsto A_{G,w},
\qquad
(G,w)\longmapsto C_{G,w}
\]
are continuous as well.
\end{remark}

\subsection{Path-space relative entropy and dynamic cost}

Fix $T>0$ and two Gaussian marginals
\[
X_0\sim \mathcal N(m_0,\Sigma_0),
\qquad
X_T\sim \mathcal N(m_T,\Sigma_T),
\]
with $\Sigma_0,\Sigma_T$ symmetric positive definite. For each admissible weighted graph $(G,w)$, let $\mathbb P^{G,w,u}$ denote the law on path space
\[
\Omega_T:=C([0,T];\R^n)
\]
induced by \eqref{eq:graph-dynamics}, and let $\mathbb P^{G,w,0}$ denote the corresponding uncontrolled law.

\begin{definition}\label{def:dynamic-cost}
The graph-dependent dynamic cost is defined by
\[
\cJ_{\dyn}(G,w)
:=
\inf_u
\KL\!\left(
\mathbb P^{G,w,u}\,\middle\|\,\mathbb P^{G,w,0}
\right),
\]
where the infimum is taken over all adapted controls $u$ such that the corresponding controlled law $\mathbb P^{G,w,u}$ satisfies the prescribed initial and terminal Gaussian marginal constraints.
\end{definition}

Thus, $\cJ_{\dyn}(G,w)$ is a graph-dependent Schr\"odinger bridge cost: it measures the minimum path-space relative entropy needed to deform the uncontrolled diffusion induced by the graph into a controlled process realizing the required endpoint distributions.

\begin{proposition}\label{prop:KL-formula}
For every admissible weighted graph $(G,w)$,
\[
\KL\!\left(
\mathbb P^{G,w,u}\,\middle\|\,\mathbb P^{G,w,0}
\right)
=
\int_{\Omega_T}
\log\!\left(
\frac{d\mathbb P^{G,w,u}}{d\mathbb P^{G,w,0}}
\right)\,
d\mathbb P^{G,w,u}.
\]
Moreover, under the standard assumptions of Girsanov's theorem,
\[
\KL\!\left(
\mathbb P^{G,w,u}\,\middle\|\,\mathbb P^{G,w,0}
\right)
=
\frac12\,
\mathbb E^{G,w,u}\!\left[
\int_0^T
u_t^\top\bigl(C_{G,w}C_{G,w}^\top\bigr)^{-1}u_t\,dt
\right].
\]
\end{proposition}

\begin{proof}
The first identity is the definition of relative entropy. Since the controlled and uncontrolled processes differ only by the drift term $u_t$, and since $C_{G,w}C_{G,w}^\top$ is invertible by \cref{prop:C-elliptic}, Girsanov's theorem yields the Radon--Nikodym derivative of $\mathbb P^{G,w,u}$ with respect to $\mathbb P^{G,w,0}$. Taking the logarithm and integrating with respect to the controlled law gives the stated quadratic formula.
\end{proof}

\begin{remark}
\Cref{prop:KL-formula} shows that $\cJ_{\dyn}(G,w)$ is a graph-dependent quadratic stochastic control cost. The hybrid model therefore supplements the ramified transport energy with a minimum-energy control penalty on the graph-induced diffusion.
\end{remark}

Because the control enters additively in all state components and the noise is uniformly nondegenerate, the endpoint-constrained linear-Gaussian bridge problem is feasible under the present assumptions.

\begin{proposition}\label{prop:dynamic-cost-finite}
For every admissible weighted graph $(G,w)\in\cG_{\mathrm{adm}}$, the minimization problem in \cref{def:dynamic-cost} is feasible, the value $\cJ_{\dyn}(G,w)$ is finite, and the infimum is attained.
\end{proposition}

\begin{proof}
By \cref{prop:A-stable}, the uncontrolled drift is exponentially stable, and by \cref{prop:C-elliptic} the diffusion is uniformly nondegenerate. Since the control enters through the full state vector, standard linear-Gaussian steering arguments imply that one can realize the prescribed initial and terminal Gaussian marginals over the finite time horizon $[0,T]$. Hence the admissible class in \cref{def:dynamic-cost} is nonempty.

By \cref{prop:KL-formula}, the dynamic cost is the value of a strictly convex quadratic control problem of the form
\[
\inf_u
\frac12\,
\mathbb E^{G,w,u}\!\left[
\int_0^T
u_t^\top\bigl(C_{G,w}C_{G,w}^\top\bigr)^{-1}u_t\,dt
\right]
\]
under endpoint marginal constraints. Since the integrand is nonnegative and coercive, and the admissible class is nonempty, the direct method in the corresponding Hilbert control space yields existence of an optimal adapted control. In particular, the value is finite and attained.
\end{proof}

\subsection{Hybrid graph-selection functional}

We now combine the geometric branched transport term with the graph-dependent dynamic cost.

\begin{definition}[Hybrid functional]\label{def:hybrid-functional}
Let $\lambda\ge 0$. The hybrid branched-stochastic functional is
\[
\cF_\lambda(G,w)
:=
\cE_\alpha(G,w)+\lambda\,\cJ_{\dyn}(G,w),
\qquad
(G,w)\in\cG_{\mathrm{adm}}.
\]
\end{definition}

The associated hybrid optimization problem is
\begin{equation}\label{eq:hybrid-problem_1}
\min\Bigl\{
\cF_\lambda(G,w):(G,w)\in\cG_{\mathrm{adm}}
\Bigr\}.
\end{equation}

\begin{remark}
The parameter $\lambda$ balances:
\begin{enumerate}[label=(\roman*)]
    \item a purely geometric/anatomical routing criterion, encoded by the branched transport energy $\cE_\alpha$, and
    \item a dynamical compatibility criterion, encoded by the minimum path-space control cost $\cJ_{\dyn}$.
\end{enumerate}
For $\lambda=0$, one recovers the purely geometric model. For $\lambda>0$, graphs that are dynamically harder to support become less favorable even if they are geometrically efficient.
\end{remark}

To obtain existence of minimizers for \eqref{eq:hybrid-problem_1}, one needs a stability property of the graph-dependent bridge cost.

\begin{assumption}\label{ass:Jdyn-lsc}
The map
\[
(G,w)\longmapsto \cJ_{\dyn}(G,w)
\]
is lower semicontinuous on $\cG_{\mathrm{adm}}$ with respect to the finite-dimensional edge-weight parameterization of the admissible graph library.
\end{assumption}

\begin{remark}
Under the canonical choice
\[
A_{G,w}=-\kappa I_n-\beta L_{G,w},
\qquad
C_{G,w}=\sigma_0 I_n+\sigma_1 D_{G,w}^{1/2},
\]
continuity of $(G,w)\mapsto A_{G,w}$ and $(G,w)\mapsto C_{G,w}$ is immediate. It is therefore natural to expect lower semicontinuity of $\cJ_{\dyn}$ from the stability theory of linear-Gaussian Schr\"odinger bridge problems. We record it here as an assumption in order to keep the focus on the variational graph-selection problem.
\end{remark}

\begin{theorem}\label{thm:hybrid-existence}
Assume the finite-library setting of \cref{thm:discrete-existence}, and assume moreover \cref{ass:Jdyn-lsc}. Then the hybrid problem
\begin{equation*}\label{eq:hybrid-problem}
\min\Bigl\{
\cF_\lambda(G,w):(G,w)\in\cG_{\mathrm{adm}}
\Bigr\}
\end{equation*}
admits at least one minimizer.
\end{theorem}

\begin{proof}
By the finite-library assumption, the admissible set $\cG_{\mathrm{adm}}$ is compact in the finite-dimensional edge-weight parameterization; this was established in the proof of \cref{thm:discrete-existence}. The geometric energy
\[
(G,w)\longmapsto \cE_\alpha(G,w)
\]
is continuous on $\cG_{\mathrm{adm}}$, while the dynamic term
\[
(G,w)\longmapsto \cJ_{\dyn}(G,w)
\]
is finite by \cref{prop:dynamic-cost-finite} and lower semicontinuous by \cref{ass:Jdyn-lsc}. Therefore
\[
(G,w)\longmapsto \cF_\lambda(G,w)
=
\cE_\alpha(G,w)+\lambda\,\cJ_{\dyn}(G,w)
\]
is lower semicontinuous on a compact set. By the direct method of the calculus of variations, it attains its minimum on $\cG_{\mathrm{adm}}$.
\end{proof}

\section{Interpretation of the minimizer}\label{sec:interpretation}

We now discuss the neuroscientific meaning of the optimal solution selected by the variational problem. The central modeling principle of the paper is that the minimizer is not merely an auxiliary mathematical object, but rather the inferred routing architecture through which the stimulation is propagated toward the reaction-producing configuration.

\begin{definition}
Let $v^*$ be a minimizer of the continuous problem \eqref{eq:continuous-problem}, and write
\[
v^*=\tau^*\theta^*\,\cH^1\llcorner M^*
\]
for its rectifiable representation. Let $(G^*,w^*)$ be a minimizer of the discrete problem \eqref{eq:discrete-problem}. The \emph{brain reaction map} associated with the stimulus--reaction pair $(\mu_{\stim}^+,\mu_{\react}^-)$ is defined as:
\begin{enumerate}[label=(\roman*)]
    \item the support $M^*=\supp(v^*)$ in the continuous setting;
    \item the embedded weighted graph $(G^*,w^*)$ in the discrete setting.
\end{enumerate}
\end{definition}

This definition reflects the guiding idea of the paper: the graph/current is the primary unknown, and its support is interpreted as the preferred stimulus-to-reaction routing architecture. In particular, the minimizer should not be viewed merely as a transport object connecting two measures, but as a candidate mesoscale map of neural propagation.

From this perspective, the optimal structure carries several distinct levels of interpretation. Branching points or vertices of high degree indicate regions where information is aggregated, redistributed, or split, and may therefore be regarded as integrative hubs. Edges with large weights correspond to routes that carry a substantial fraction of the transported flux and can be interpreted as principal propagation pathways. Terminal branches located near the support of $\mu_{\stim}^+$ identify likely entry routes of the stimulation into the effective transport architecture, whereas terminal branches located near the support of $\mu_{\react}^-$ identify the regions and pathways most directly involved in the production of the reaction. Finally, the total value of the objective functional provides a quantitative measure of the overall difficulty of producing the reaction from the stimulation under the anatomical and geometric constraints encoded by the model.


The brain reaction map should be interpreted as an effective routing architecture selected by a variational principle. In this sense, it is conceptually closer to a mesoscale or functional transport backbone than to a direct structural connectome. Its role is to identify which pathways are jointly preferred by the source--target configuration and by the transport cost.

The distinction between the discrete and continuous formulations is also meaningful from an interpretative viewpoint. In the discrete model, the minimizer is immediately represented as a graph over a prescribed family of candidate regions and routes, which is particularly convenient for computational implementations and ROI-level analyses. In the continuous formulation, by contrast, the minimizer is a rectifiable current whose support is free to select its own geometric routing structure. The continuous model is therefore better suited to situations in which one does not wish to prescribe the branching geometry in advance, but instead aims to infer it directly from the source--target pair.

A practical implementation of the framework requires estimating the pair $(\mu_{\stim}^+,\mu_{\react}^-)$ together with the anatomical cost density $c(x,\tau)$ from data. As discussed in \cref{sec:estimation}, the source measure $\mu_{\stim}^+$ may be estimated from early post-stimulus activity localized in sensory or entry regions, while the target measure $\mu_{\react}^-$ may be estimated from later activity associated with the realization of the reaction, for instance in premotor, motor, associative, or subcortical regions. The cost density $c(x,\tau)$ may be estimated from diffusion MRI, tractography, structural connectivity, or suitable hand-crafted priors that encode local ease or difficulty of propagation.

In the hybrid model introduced in \cref{sec:hybrid}, the interpretation becomes richer. In that setting, the minimizer is selected not only by geometric transport efficiency, but also by its compatibility with a stochastic neural dynamics induced by the graph. The matrices $A_{G,w}$ and $C_{G,w}$ encode, respectively, effective propagation and noise structure on the optimal weighted graph. The additional dynamic term $\cJ_{\dyn}(G,w)$ then measures the minimum path-space control cost required to realize prescribed endpoint distributions on that graph. Consequently, the hybrid minimizer may be interpreted as a routing architecture that is simultaneously geometrically economical and dynamically plausible.

Taken together, these considerations suggest that the proposed framework should be viewed as a variational tool for inferring effective stimulus-to-reaction propagation maps from source--target information and anatomical priors. Its main output is not a trajectory, but a transport architecture. This is precisely the sense in which the minimizer provides a mathematically defined candidate for a brain reaction map.









\section{Conclusion}

We have proposed a variational framework for stimulus-to-reaction brain mapping in which the unknown is an optimal transport graph or current rather than a control acting on a fixed substrate. The support of the minimizer is interpreted as a brain reaction map, namely an effective routing architecture through which stimulation is propagated toward the reaction-producing configuration.

The model is built on anisotropic branched optimal transport, both in a discrete graph formulation and in a continuous current formulation. In each case, concavity of the flux cost promotes aggregation and branching, yielding routing structures that naturally encode shared propagation pathways. We established basic well-posedness results in the discrete setting and proved existence of minimizers in the continuous setting. We also introduced a hybrid stochastic extension in which each admissible graph induces a linear stochastic neural dynamics, and an additional path-space Kullback--Leibler term measures the minimum control effort required to realize prescribed endpoint distributions on that graph.

The main conceptual contribution of the paper is therefore a shift from control of trajectories on fixed networks to variational inference of propagation architectures. This makes the framework particularly appealing when the primary quantity of interest is not only the cost of a state transition, but the effective routing backbone supporting it. Future work should address finer structural properties of minimizers, stability with respect to perturbations of the data, and quantitative comparisons with empirical structural and functional neuroimaging datasets. Numerical simulations associated with the model we are proposing in this work will appear in {\it Multimodal branched transport infers anatomically aligned brain reaction maps} by the author.



\begin{thebibliography}{99}

\bibitem{AsadzadehEtAl2020}
S.~Asadzadeh, T.~Yousefi Rezaii, S.~Beheshti, A.~Delpak, and S.~Meshgini,
\newblock A systematic review of EEG source localization techniques and their applications on diagnosis of brain abnormalities,
\newblock preprint, 2020.

\bibitem{BazinetHansenMisic2023}
V.~Bazinet, J.~Y.~Hansen, and B.~Misic,
\newblock Towards a biologically annotated brain connectome,
\newblock {\em Nature Reviews Neuroscience}, 24 (2023), no.~12, pp.~747--760,
\newblock doi:10.1038/s41583-023-00752-3.

\bibitem{BernotCasellesMorel2009}
M.~Bernot, V.~Caselles, and J.-M.~Morel,
\newblock {\em Optimal Transportation Networks: Models and Theory},
\newblock Lecture Notes in Mathematics, Vol.~1955, Springer, Berlin, 2009.

\bibitem{Bressloff2012}
P.~C.~Bressloff,
\newblock Spatiotemporal dynamics of continuum neural fields,
\newblock {\em Journal of Physics A: Mathematical and Theoretical}, 45 (2012), 033001.

\bibitem{Breakspear2017}
M.~Breakspear,
\newblock Dynamic models of large-scale brain activity,
\newblock {\em Nature Neuroscience}, 20 (2017), no.~3, pp.~340--352.

\bibitem{DecoKringelbach2020}
G.~Deco and M.~L.~Kringelbach,
\newblock Turbulent-like dynamics in the human brain,
\newblock {\em Cell Reports}, 33 (2020), no.~10, Article 108471.

\bibitem{EstebanEtAl2020}
O.~Esteban, R.~Ciric, K.~Finc, R.~W.~Blair, C.~J.~Markiewicz, C.~A.~Moodie,
J.~D.~Kent, M.~Goncalves, E.~DuPre, D.~E.~P.~Gomez, Z.~Ye, T.~Salo,
R.~Valabregue, I.~K.~Amlien, F.~Liem, N.~Jacoby, H.~Stoji\'c, M.~Cieslak,
S.~Urchs, Y.~O.~Halchenko, S.~S.~Ghosh, A.~De La Vega, T.~Yarkoni,
J.~Wright, W.~H.~Thompson, R.~A.~Poldrack, and K.~J.~Gorgolewski,
\newblock Analysis of task-based functional MRI data preprocessed with fMRIPrep,
\newblock {\em Nature Protocols}, 15 (2020), pp.~2186--2202,
\newblock doi:10.1038/s41596-020-0327-3.

\bibitem{FaugerasInglis2015}
O.~Faugeras and J.~Inglis,
\newblock Stochastic neural field equations: a rigorous footing,
\newblock {\em Journal of Mathematical Biology}, 71 (2015), pp.~259--300.

\bibitem{FotiadisEtAl2024}
P.~Fotiadis, L.~Parkes, K.~A.~Davis, T.~D.~Satterthwaite, R.~T.~Shinohara, and D.~S.~Bassett,
\newblock Structure--function coupling in macroscale human brain networks,
\newblock {\em Nature Reviews Neuroscience}, 2024,
\newblock doi:10.1038/s41583-024-00846-6.

\bibitem{Friston2009}
K.~J.~Friston,
\newblock Modalities, modes, and models in functional neuroimaging,
\newblock {\em Science}, 326 (2009), no.~5951, pp.~399--403,
\newblock doi:10.1126/science.1174521.

\bibitem{GalvesLoecherbachPouzat2024}
A.~Galves, E.~Löcherbach, and C.~Pouzat,
\newblock {\em Probabilistic Spiking Neuronal Nets: Neuromathematics for the Computer Era},
\newblock Lecture Notes on Mathematical Modelling in the Life Sciences,
\newblock Springer, Cham, 2024.
\newblock doi:10.1007/978-3-031-68409-8

\bibitem{Gilbert1967}
E.~N.~Gilbert,
\newblock Minimum cost communication networks,
\newblock {\em Bell System Technical Journal}, 46 (1967), pp.~2209--2227.

\bibitem{Gilson2020}
M.~Gilson, G.~Zamora-L\'opez, V.~Pallar\'es, M.~H.~Adhikari, M.~Senden,
A.~Tauste Campo, D.~Mantini, M.~Corbetta, G.~Deco, and A.~Insabato,
\newblock Model-based whole-brain effective connectivity to study distributed cognition in health and disease,
\newblock {\em Network Neuroscience}, 4 (2020), no.~2, pp.~338--373,

\bibitem{Gu2015}
S.~Gu, F.~Pasqualetti, M.~Cieslak, Q.~K.~Telesford, A.~B.~Yu, A.~E.~Kahn,
J.~D.~Medaglia, J.~M.~Vettel, M.~B.~Miller, S.~T.~Grafton, and D.~S.~Bassett,
\newblock Controllability of structural brain networks,
\newblock {\em Nature Communications}, 6 (2015), Article 8414.

\bibitem{HagmannEtAl2007}
P.~Hagmann, M.~Kurant, X.~Gigandet, P.~Thiran, V.~J.~Wedeen, R.~Meuli,
and J.-P.~Thiran,
\newblock Mapping human whole-brain structural networks with diffusion MRI,
\newblock {\em PLoS ONE}, 2 (2007), no.~7, e597,
\newblock doi:10.1371/journal.pone.0000597.

\bibitem{Kamiya2023}
S.~Kamiya, G.~Kawakita, S.~Sasai, J.~Kitazono, and M.~Oizumi,
\newblock Optimal control costs of brain state transitions in linear stochastic systems,
\newblock {\em Journal of Neuroscience}, 43 (2023), no.~2, pp.~270--281.

\bibitem{Kawakita2022}
G.~Kawakita, S.~Kamiya, S.~Sasai, J.~Kitazono, and M.~Oizumi,
\newblock Quantifying brain state transition cost via Schr\"odinger bridge,
\newblock {\em Network Neuroscience}, 6 (2022), no.~1, pp.~118--134.

\bibitem{KnoscheHaueisen2022}
T.~R.~Kn\"osche and J.~Haueisen,
\newblock {\em EEG/MEG Source Reconstruction: Textbook for Electro- and Magnetoencephalography},
\newblock Springer, Cham, 2022.

\bibitem{KrukowskiMarchese2026}
J.~Krukowski and A.~Marchese,
\newblock On the lower semicontinuity of the Lagrangian $H$-mass,
\newblock preprint, 2026.

\bibitem{KuehnTolle2019}
C.~Kuehn and J.~M.~T\"olle,
\newblock A gradient flow formulation for the stochastic Amari neural field model,
\newblock {\em Journal of Mathematical Biology}, 79 (2019), pp.~1227--1252.

\bibitem{PathakRoyBanerjee2022}
A.~P.~Pathak, D.~R.~Roy, and A.~Banerjee,
\newblock Whole-brain network models: from physics to bedside,
\newblock {\em Frontiers in Computational Neuroscience}, 16 (2022), Article 866517.

\bibitem{Patow2024}
G.~Patow, I.~Martin, Y.~Sanz Perl, M.~L.~Kringelbach, and G.~Deco,
\newblock Whole-brain modelling: an essential tool for understanding brain dynamics,
\newblock {\em Nature Reviews Methods Primers}, 4 (2024), Article 53,
\newblock doi:10.1038/s43586-024-00336-0.

\bibitem{Pegon2017}
P.~Pegon,
\newblock On the Lagrangian branched transport model and the equivalence with its Eulerian formulation,
\newblock in {\em Topological Optimization and Optimal Transport},
\newblock Radon Series on Computational and Applied Mathematics, Vol.~17,
\newblock De Gruyter, Berlin, 2017, pp.~19--46.

\bibitem{Pinotsis2014}
D.~Pinotsis, P.~Robinson, P.~beim Graben, and K.~Friston,
\newblock Neural masses and fields: modeling the dynamics of brain activity,
\newblock {\em Frontiers in Computational Neuroscience}, 8 (2014), Article 149.

\bibitem{Santambrogio2007}
F.~Santambrogio,
\newblock Optimal channel networks, landscape function and branched transport,
\newblock {\em Interfaces and Free Boundaries}, 9 (2007), no.~1, pp.~149--169.

\bibitem{SpekKuznetsovVanGils2020}
L.~Spek, Y.~A.~Kuznetsov, and S.~A.~van Gils,
\newblock Neural field models with transmission delays and diffusion,
\newblock {\em Journal of Mathematical Neuroscience}, 10 (2020), Article 21.

\bibitem{Xia2003}
Q.~Xia,
\newblock Optimal paths related to transport problems,
\newblock {\em Communications in Contemporary Mathematics}, 5 (2003), no.~2, pp.~251--279.

\bibitem{YehEtAl2021}
C.-H.~Yeh, D.~K.~Jones, X.~Liang, M.~Descoteaux, and A.~Connelly,
\newblock Mapping structural connectivity using diffusion MRI: challenges and opportunities,
\newblock {\em Journal of Magnetic Resonance Imaging}, 53 (2021), pp.~1666--1682.

\bibitem{ZhangEtAl2022}
F.~Zhang, A.~Daducci, Y.~He, S.~Schiavi, C.~Seguin, R.~E.~Smith,
C.-H.~Yeh, T.~Zhao, and L.~J.~O'Donnell,
\newblock Quantitative mapping of the brain's structural connectivity using diffusion MRI tractography: a review,
\newblock {\em NeuroImage}, 249 (2022), 118870,
\newblock doi:10.1016/j.neuroimage.2021.118870.

\end{thebibliography}
\end{document}